\documentclass[twoside,a4paper,reqno,12pt]{amsart}
\usepackage{lmodern}

\usepackage{cite}
\usepackage{mathrsfs}
\usepackage{tabularx} 
\usepackage{booktabs} 
\usepackage{amsmath}
\usepackage{amsthm}
\usepackage{latexsym}
\usepackage{amssymb,bm}
\usepackage{amsfonts}
\usepackage{longtable}
\usepackage{multirow}
\usepackage[colorlinks,linkcolor=cyan,anchorcolor=blue,citecolor=red,backref=page]{hyperref}
\usepackage{braket}
\usepackage{amsmath,enumerate}
\usepackage[top=28mm,right=25mm,bottom=28mm,left=25mm]{geometry}
\usepackage{todonotes}
\usepackage{xcolor}

\usepackage{graphics}

\usepackage{color}

\def\Aut{{\sf Aut}} 
 \def\S{\rm S}

\newtheorem{thm}{Theorem}[section]

\newtheorem{proposition}[thm]{Proposition}

\newtheorem{theorem}[thm]{Theorem}
\newtheorem{lemma}[thm]{Lemma}

\newtheorem{corollary}[thm]{Corollary}
\newtheorem{problem}{Problem}

\theoremstyle{definition}
\newtheorem{defi}[thm]{Definition}
\newtheorem{remark}[thm]{Remark}
\newtheorem*{remark*}{Remarks}
\newtheorem{example}[thm]{Example}
\newtheorem{construction}[thm]{Construction}

\newcommand{\emptycommand}{}

\def\S{{\mathrm{S}}}

\def\Sym{{\mathrm{Sym}}}

\def\leq{\leqslant}
\def\le{\leqslant}
\def\geq{\geqslant}
\def\ge{\geqslant}

\keywords{De Bruijn graph, Cayley graph, Praeger-Xu graph}
\subjclass[2000]{05C38, 20B25}

\begin{document}
	
	\title[Generalizing Praeger--Xu graphs]{Generalizing Praeger--Xu graphs}
	
	\author{Zhaochen Ding}
    \author{Ci Xuan Wu}
	\address{Z. Ding, The Faculty of Mathematics, Natural Sciences and Information Technologies\\
		University of Prismorska\\
		Koper\\
		Slovenia}
	\email{zhaochen.ding@upr.si}

	\address{C. X. Wu, School of Statistics and Mathematics \\
		Yunnan University of Finance and Economics\\
		Kunming 650021\\
		P. R. China}
	\email{wucixuan@gmail.com }

	\date{}

	\begin{abstract}
We uniformly study directed Praeger--Xu graphs and Praeger's generalized family $\overrightarrow{\operatorname{PX}}(\Gamma,r,s)$, without requiring that $s\mid r$ and $\Gamma$ is an undirected graph. 
Assuming that $\Gamma\times\overrightarrow{K}_2$ is connected, we determine the full automorphism group of $\overrightarrow{\operatorname{PX}}(\Gamma,r,s)$ for every $r$ and $s$, prove connectivity, characterize vertex-transitivity, and give an equivalent criterion for $t$-arc-transitivity. We also study when $\overrightarrow{\operatorname{PX}}(\Gamma,r,s)$ is a Cayley graph.

We also give two applications. First, we give a new family of twin-free connected arc-transitive digraphs of valency $4$ whose vertex stabilizers have exponential growth. 
We also give a connected regular digraph that is not vertex-transitive but whose direct product with a directed cycle is Cayley.
\end{abstract}
	\maketitle
	
	\section{Introduction}

\begingroup
\hypersetup{linkcolor=blue,citecolor=blue,urlcolor=blue}

All graphs and digraphs in this paper are finite. A \emph{digraph} $\Gamma$ has a vertex set $V\Gamma$ and an arc set $A\Gamma\subseteq V\Gamma\times V\Gamma$. We view a graph as a digraph with a symmetric arc set. A digraph is \emph{asymmetric} if it has no pair of opposite arcs. The \emph{underlying graph} of a digraph is obtained by replacing each arc $(u,v)$ with the undirected edge $\{u,v\}$. Throughout, \emph{connected} means that the underlying graph is connected. Strong connectivity is not considered in this paper. A digraph is \emph{regular of valency} $k$ if every vertex has in-valency and out-valency $k$. Throughout, we assume that every vertex of $\Gamma$ has positive in- and out-valency. For $G\le\Aut(\Gamma)$, we say that $\Gamma$ is \emph{$G$-vertex-transitive} or \emph{$G$-arc-transitive} if $G$ is transitive on vertices or arcs, respectively, and omit $G$ when $G=\Aut(\Gamma)$. We write $G_v$ for the stabilizer of a vertex $v$.

Praeger and Xu introduced the Praeger--Xu graphs while studying arc-transitive graphs of valency $2p$, where $p$ is prime \cite{PX}. 
Jajcay et al. characterize the tetravalent Praeger--Xu graphs that are Cayley graphs by a polynomial condition on their parameters \cite[Theorem~1.1]{Jajcay2022Cayleyness}. 
Poto\v{c}nik et al. show that, apart from the Praeger--Xu family, vertex stabilizers in connected tetravalent arc-transitive graphs satisfy a sublinear bound in terms of the number of vertices \cite[Corollary~1.3]{Potocnik2015Bounding}. In particular, arc-transitive tetravalent Praeger--Xu graphs have vertex stabilizers whose orders grow exponentially in the number of vertices \cite{Potocnik2015Bounding}. 
We use \emph{exponential growth} to refer to this growth of vertex-stabilizer orders.

Praeger later introduced a generalized Praeger--Xu construction and used it to construct $t$-arc-transitive digraphs for every $t\geq1$ \cite{Praeger1989high}. This contrasts with Weiss's bound for finite connected undirected graphs of valency at least $3$ \cite{Weiss1981}. Her construction assumes that $s\mid r$ and that the base graph is undirected. She deal with Praeger--Xu digraphs and generalized Praeger--Xu digraphs separately, rather than by one uniform method addressing both the original and generalized Praeger--Xu graphs.

M\"oller, Poto\v{c}nik and Seifter \cite{Moller2019Sharply} give a more general construction in a different direction. Their digraph $\mathbb{Z}(\Delta,k)$ is generally infinite, and its finite cyclic quotients include Praeger's generalized Praeger--Xu digraphs. Under suitable hypotheses, they obtain sharply $s$-arc-transitive examples, that is, examples that are $s$-arc-transitive but not $(s+1)$-arc-transitive. The periods of their finite quotients are multiples of $s$, and their work does not give a general formula for the full automorphism group.



For a digraph $\Gamma$, integers $r\geq3$ and $s\geq1$, Construction~\ref{construction} defines
\[
\overrightarrow{\operatorname{PX}}(\Gamma,r,s)=\Gamma^{*s}\times\overrightarrow{C}_r.
\]
This definition has no divisibility condition on $r$ and $s$, and $\Gamma$ may be any finite digraph with positive in- and out-valency. When $\Gamma$ is the complete digraph with self-loops on $m$ vertices, it gives the directed Praeger--Xu graph $\overrightarrow{\operatorname{PX}}(m,r,s)$; its underlying graph is $\operatorname{PX}(m,r,s)$. Thus it gives a uniform description of the directed Praeger--Xu graphs and Praeger's generalized family. We write $\Gamma[r]$ for $\overrightarrow{\operatorname{PX}}(\Gamma,r,r)$.

Our main results require $\Gamma\times\overrightarrow{K}_2$ to be connected. For an undirected graph $\Gamma$, this is equivalent to $\Gamma$ being connected and non-bipartite. The same hypothesis appears in the work of M\"oller, Poto\v{c}nik and Seifter \cite{Moller2019Sharply}: connectedness of the canonical double half-cover is used to identify the alternets and obtain connected examples. In our setting, Remark~\ref{rmk:necessarity-connected-alternets} gives an example where the description of the automorphism group fails without this hypothesis.

Under this hypothesis, we determine $\Aut(\overrightarrow{\operatorname{PX}}(\Gamma,r,s))$ for every $r$ and $s$ in terms of the two-fold automorphisms of $\Gamma$ (Theorem~\ref{thm:aut-general-px}). We also prove that the resulting digraph is connected, determine vertex-transitivity, and give an equivalent criterion for $t$-arc-transitivity (Theorem~\ref{cor:nice-connectivity}, Theorem~\ref{thm:arc-transitivity-new}). 
We also give a criterion for Cayleyness when all two-fold automorphisms of the base digraph are diagonal (Theorem~\ref{thm:Cayleyness}). We apply these results to study $s$-arc-transitivity of certain generalized Praeger--Xu digraphs (Corollary~\ref{cor:arc-transitivity-complete-loops}, Corollary~\ref{cor:arc-transitivity-complete-loopless}, Corollary~\ref{cor:arc-transitivity-complete-multipartite}), and construct a family of twin-free connected arc-transitive digraphs of valency $4$ whose vertex stabilizers have exponential growth (Theorem~\ref{thm:exponential-growth}). We also give a connected regular digraph that is not vertex-transitive although its direct product with a directed cycle is Cayley (Example~\ref{examp:non-Cayley}). 
This contrasts with the corresponding behaviour of connected direct products of undirected graphs \cite[Theorem~8.19]{Hammack2011Handbook}. As a byproduct, we prove that $\Aut(\Gamma^{*s})\cong\Aut(\Gamma)$ (Theorem~\ref{thm:auto-g-de-bruijn}). 

The construction leaves several questions.
\begin{problem}
    Can the hypothesis that $\Gamma\times\overrightarrow{K}_2$ is connected be removed from the automorphism and arc-transitivity results?
\end{problem}
\begin{problem}
    Determine when $\overrightarrow{\operatorname{PX}}(\Gamma,r,s)$ is a Cayley digraph.
\end{problem}
\begin{problem}
    Find loopless $\Gamma$, integers $r$ and $s$ with $s\nmid r$ such that $\overrightarrow{\operatorname{PX}}(\Gamma,r,s)$ is vertex-transitive.
\end{problem}

Throughout, group actions are written on the right. For $g,h\in G$, we write $g^h=h^{-1}gh$, and $\mathbb{Z}_k$ is written additively. We use $G\wr H$ for the regular wreath product of $G$ by $H$.
When the context is unambiguous, we also write the integer $i$ to represent its congruence class modulo $k$.

	\section{Preliminary}\label{sec:preliminary}

In this section, we recall some terminology for digraphs.

\subsection{Two-fold automorphism}
The concept of a two-fold automorphism originates from the study of the canonical double cover of a graph.
It was first formally introduced by Lauri, Mizzi, and Scapellato~\cite{LauriMS11} as a generalization of the classical notion of a graph automorphism, building on their earlier investigation of two-fold orbital digraphs~\cite{LauriMS04}.
Two-fold automorphisms are connected to the notion of graph stability: a graph is stable if and only if it admits no non-diagonal two-fold automorphism~\cite[Theorem 3.2]{LauriMS15}.
Recent work has further explored the realizability of prescribed two-fold automorphism groups via the canonical double cover construction~\cite{Bychawski2024}.

\begin{defi}
Let $\Gamma$ be a digraph with vertex set $V\Gamma$ and arc set $A\Gamma$, and let $\sigma_1, \sigma_2 \in \operatorname{Sym}(V(\Gamma))$. A pair $(\sigma_1, \sigma_2)$ is called a \emph{two-fold automorphism} of $\Gamma$ if 
\[
(u, v) \in A\Gamma \iff (u\sigma_1, v\sigma_2) \in A\Gamma \quad \text{for all } u, v \in V\Gamma.
\]
The set of all two-fold automorphisms of $\Gamma$ forms a subgroup of $\operatorname{Sym}(V\Gamma)^2$, called the \emph{two-fold automorphism group} of $\Gamma$ and denoted by $\operatorname{Aut}^{\mathrm{TF}}(\Gamma)$.
\end{defi}

The pair $(\operatorname{Sym}(V\Gamma), \operatorname{Aut}^{\mathrm{TF}}(\Gamma))$ defines a digraph 
with vertex set $\operatorname{Sym}(V\Gamma)$ and arc set $\operatorname{Aut}^{\mathrm{TF}}(\Gamma)$. 
When the context is clear, a two-fold automorphism $(\sigma_1, \sigma_2)$ 
is denoted by $\sigma_1 \to \sigma_2$. 
The following example is applied in Section~\ref{sec:exponential-growth}.

\begin{example}\label{examp:PX}
	Let $b\ge 1$, let $\ell\ge 2$, and let $\Gamma=K_{b,\ldots,b}$ be the complete multipartite graph with $\ell$ parts of size $b$.
	Then $V\Gamma=\mathbb{Z}_b\times\mathbb{Z}_\ell$ and $(x,y)$ is adjacent to $(z,t)$ if and only if $t\ne y$.
	For $i\in\mathbb{Z}_\ell$, let $B_{\emptycommand{i}}=\{(x,\emptycommand{i})\in \mathbb{Z}_b\times\mathbb{Z}_\ell\mid x\in\mathbb{Z}_b\}$.
	Let $H=\prod_i\operatorname{Sym}(B_{\emptycommand{i}})\le\operatorname{Sym}(V\Gamma)$.
	For every $\sigma\in \operatorname{Sym}(\mathbb{Z}_\ell)$, define $\sigma^*\in \operatorname{Sym}(V\Gamma)$ by
	\[(x,\emptycommand{i})\sigma^*=(x,\emptycommand{i}\sigma).\]
	Let $K=\{(\sigma^*,\sigma^*)\mid \sigma\in \operatorname{Sym}(\mathbb{Z}_\ell)\}$.
	Then
	\[
		\Aut^{\operatorname{TF}}(\Gamma)=(H\times H)\rtimes K
		\cong (\operatorname{S}_b^{\ell}\times\operatorname{S}_b^{\ell})\rtimes\operatorname{S}_\ell.
	\]
	For $b=2$ and $\ell=3$, $\Aut^{\operatorname{TF}}(\Gamma)
		\cong \mathbb{Z}_2^{6}\rtimes\operatorname{S}_3$. 
    As we shall see in Section \ref{sec:construction}, $\Gamma\cong \operatorname{PX}(2,3,1)$.
\end{example}
\begin{proof}
	Let $(\alpha,\beta)\in\Aut^{\operatorname{TF}}(\Gamma)$. For each
	$i\in\mathbb{Z}_\ell$, the set $B_{\emptycommand{i}}\alpha$ is contained in one
	part of the multipartition. Indeed, vertices in $B_{\emptycommand{i}}$ have the
	same neighbourhood, and the definition of a two-fold automorphism implies
	that their images under $\alpha$ have the same neighbourhood. The same argument applies to $\beta$.
	Thus there are permutations $p,q\in\operatorname{Sym}(\mathbb{Z}_\ell)$ such that
	\[
		B_{\emptycommand{i}}\alpha=B_{\emptycommand{i}p}
		\quad\text{and}\quad
		B_{\emptycommand{i}}\beta=B_{\emptycommand{i}q}.
	\]
	We claim that $p=q$.
	Let $\emptycommand{i}pq^{-1}=\emptycommand{j}$. Then $B_{\emptycommand{i}}\alpha=B_{\emptycommand{i}p}=B_{\emptycommand{j}q}=B_{\emptycommand{j}}\beta$.
	If $\emptycommand{i}\ne \emptycommand{j}$, then every vertex of $B_{\emptycommand{i}}$ is adjacent to every vertex of $B_{\emptycommand{j}}$,
	and hence, by the definition of a two-fold automorphism, every vertex of $B_{\emptycommand{i}}\alpha$ is adjacent to every vertex of $B_{\emptycommand{j}}\beta$.
	This implies that $\emptycommand{i}p\ne \emptycommand{j}q$, a contradiction.
	Therefore, we have $p=q$. Let $\sigma=p=q$. Then
	\[
		(\alpha,\beta)(\sigma^*,\sigma^*)^{-1}
		=(h_1,h_2)
	\]
	for some $h_1,h_2\in H$.
	The reverse inclusion is immediate, so the displayed equality follows.
	Since $H\cong\operatorname{S}_b^\ell$ and $K\cong\operatorname{S}_\ell$ acts by
	permuting the $\ell$ factors in each copy of $H$, the claimed isomorphism	follows.
\end{proof}

\subsection{Alternets and loosely attached digraphs}

    Let $\Gamma$ be a digraph, and let $(u,v)\in A\Gamma$. We call $u$ the \emph{tail} of $(u,v)$ and $v$ the \emph{head} of $(u,v)$. A vertex is a \emph{sink} if it is the head of an arc and not the tail of any arc. A vertex is a \emph{source} if it is the tail of an arc and not the head of any arc.
    A \emph{successor} of an arc $(u,v)$ is an arc $(v,w)$.
	Two arcs are \emph{related} if they have a common head or a common tail, and we denote by $R$ the transitive closure of this relation. 
    
    The \emph{alternet} of an arc $a$ is the subdigraph induced by $R(a)$. For convenience, we sometimes identify this subdigraph with the equivalence class $R(a)$ if the context is clear. An alternet is \emph{degenerate} if it contains a successor of one of its arcs.
    Clearly, a non-degenerate alternet is a bipartite digraph where one part consists of sources and the other part consists of sinks.
    For convenience, we simply call such a digraph a \emph{naturally directed bipartite graph}.
	The digraph $\Gamma$ is \emph{loosely attached} if it has no degenerate alternets and $|\mathrm{sink}(B_1)\cap\mathrm{source}(B_2)|\le 1$ for any two alternets $B_1,B_2$.  The \emph{digraph of alternets} $\mathsf{Al}(\Gamma)$ has alternets as vertices and arc $(B_1,B_2)$ if and only if $\mathrm{sink}(B_1)\cap\mathrm{source}(B_2)\neq\emptyset$. The \emph{partial line digraph} $\mathsf{Pl}(\Gamma)$ is obtained by choosing the set of arcs as vertices and defining $(a,b)$ as an arc whenever $b$ is a successor of $a$. We point out that our definition of $\mathsf{Pl}(\Gamma)$ varies slightly from that in \cite{Potocnik2010}, but these are equivalent if $\Gamma$ is asymmetric. For further details, see \cite{Potocnik2010}.

	Now we collect some results in \cite{Potocnik2010} that will be used in Section~\ref{sec:construction}.

	\begin{proposition}\label{prop:alternets}
		Let $\Gamma$ be a digraph.
		\begin{itemize}
			\item [\rm (i)] \cite[Lemma 3.1 {\rm (i)}]{Potocnik2010} If $\Gamma$ is loosely attached and $G\le\Aut(\Gamma)$, then $G\le\Aut(\mathsf{Al}(\Gamma))$.
			\item [\rm (ii)] \cite[Lemma 3.1 {\rm (ii)}, Lemma 3.2 {\rm (ii), (iii)}]{Potocnik2010} If $\Gamma$ is asymmetric without degenerate alternets, then $\mathsf{Al}(\mathsf{Pl}(\Gamma))=\Gamma$ and $\Aut(\Gamma)\le\Aut(\mathsf{Pl}(\Gamma))$.
			\item [\rm (iii)] \cite[Lemma 3.1 {\rm (ii)}, Lemma 3.3]{Potocnik2010} If $\Gamma$ is asymmetric and loosely attached, and every alternet of $\Gamma$ is complete bipartite, then $\mathsf{Pl}(\mathsf{Al}(\Gamma))=\Gamma$.
		\end{itemize}
	\end{proposition}
	
	\section{Generalized Praeger--Xu graphs}\label{sec:construction}
	
	A directed \emph{Praeger--Xu graph} $\overrightarrow{\mathrm{PX}}(v,r,s)$ is the digraph on $\mathbb{Z}_v^s\times\mathbb{Z}_r$ with arcs
    $$(x_1,\ldots,x_s,i)\to(x_2,\ldots,x_s,y,i+1),$$
    and its undirected version $\mathrm{PX}(v,r,s)$ is its underlying graph (see \cite{PX,Praeger1989high}).

	Now we generalize the construction of Praeger--Xu graphs.
	For the remainder, when the meaning is clear, we write $u\to v$ to mean that $(u,v)$ is an arc.

	\begin{construction}
    \label{construction}
  Let $\Gamma$ be a digraph and let $s\ge 1$.
  Define a digraph $\Gamma^{*s}$ as follows:
  \begin{itemize}
    \item [\rm (i)] $V\Gamma^{*s}=(V\Gamma)^s$; and
    \item [\rm (ii)]
      $(v_1,\ldots,v_s)\to(w_1,\ldots,w_s)$ if and only if 
      \[
        (v_1,w_s)\in A\Gamma\quad\text{and}\quad w_i=v_{i+1}\ \text{for all } i\in\{1,\ldots,s-1\}.
      \]
  \end{itemize}
  In particular, $\Gamma^{*1}=\Gamma$.
  For $r\ge 3$, define $\overrightarrow{\operatorname{PX}}(\Gamma,r,s)=\Gamma^{*s}\times\overrightarrow{C_r}$ and define that $\operatorname{PX}(\Gamma,r,s)$ is the underlying graph of $\overrightarrow{\operatorname{PX}}(\Gamma,r,s)$.
  When $r=s$, we write $\Gamma[r]$ for this digraph and write $\Gamma\{r\}$ for the underlying graph of $\Gamma[r]$.
\end{construction}

	\begin{example}
    A \emph{De Bruijn digraph} $\overrightarrow{B}(m,n)$ is defined as the digraph on $[m]^n$ with arcs $(x_1,\ldots,x_n)\to(x_2,\ldots,x_n,y)$.
  Let $\Gamma$ be the digraph with vertex set $V\Gamma$ and arc set $V\Gamma\times V\Gamma$, and let $m=|V\Gamma|$.
  Then for $s\ge 1$, $\Gamma^{*s}=\overrightarrow{B}(m,s)=\mathsf{Pl}^{s-1}(\Gamma)$.
  One can also check that the directed Praeger-Xu graph $\overrightarrow{\mathrm{PX}}(m,r,s)=\Gamma^{*s}\times\overrightarrow{C}_r=\overrightarrow{\mathrm{PX}}(\Gamma,r,s)$ for $r\ge 3$.
  The underlying graph of $\overrightarrow{\operatorname{PX}}(m,r,s)$ is exactly the Praeger-Xu graph $\operatorname{PX}(m,r,s)$.
  Thus, Construction~\ref{construction} extends the family of directed Praeger-Xu graphs, and  also extends the family of Praeger-Xu graphs.
\end{example}

Now we introduce the cycle-shift action.

\begin{defi}\label{def:twist-action}
  Let $G$ act on a set $\Omega$, and let $s,r\ge 1$.
  We define an action of $G\wr\mathbb{Z}_r$ on $\Omega^s\times\mathbb{Z}_r$ by
  \[
    (w_1,\ldots,w_s,\emptycommand{i})(g_1,\ldots,g_r,\emptycommand{j})
    =(w_1g_{i+1},\ldots,w_sg_{i+s},\emptycommand{i+j}),
  \]
  for all $(w_1,\ldots,w_s,\emptycommand{i})\in \Omega^s\times\mathbb{Z}_r$ and
  $(g_1,\ldots,g_r,\emptycommand{j})\in G\wr\mathbb{Z}_r$ (indices modulo $r$).
  We call this the \emph{cycle-shift action} of
  $G\wr\mathbb{Z}_r$ on $\Omega^s\times\mathbb{Z}_r$.
\end{defi}

\begin{lemma}
  The cycle-shift action in Definition~\ref{def:twist-action} is well-defined.
\end{lemma}
\begin{proof}
  Let $(g_1,\ldots,g_r,\emptycommand{j}), (h_1,\ldots,h_r,\emptycommand{\ell})\in G\wr\mathbb{Z}_r$.
  Using indices modulo $r$, their product is
  \[
    (g_1,\ldots,g_r,\emptycommand{j})(h_1,\ldots,h_r,\emptycommand{\ell})
    =(g_1h_{1+j},\ldots,g_rh_{r+j},\emptycommand{j+\ell}).
  \]
  For $(w_1,\ldots,w_s,\emptycommand{i})\in\Omega^s\times\mathbb{Z}_r$, we compute
  \[\begin{aligned}
    &\bigl((w_1,\ldots,w_s,\emptycommand{i})(g_1,\ldots,g_r,\emptycommand{j})\bigr)
    (h_1,\ldots,h_r,\emptycommand{\ell})\\
    =&(w_1g_{i+1},\ldots,w_sg_{i+s},\emptycommand{i+j})(h_1,\ldots,h_r,\emptycommand{\ell})\\
    =&(w_1g_{i+1}h_{i+1+j},\ldots,w_sg_{i+s}h_{i+s+j},\emptycommand{i+j+\ell}),
  \end{aligned}
  \]
  which equals
  \[
    (w_1,\ldots,w_s,\emptycommand{i})
    (g_1h_{1+j},\ldots,g_rh_{r+j},\emptycommand{j+\ell}).
  \]
  Thus the action is associative. The identity element
  $(1,\ldots,1,\emptycommand{0})$ acts trivially, so this is a well-defined group action.
\end{proof}

\subsection{Basic properties}

In this section, we give some basic properties regarding $\overrightarrow{\operatorname{PX}}(\Gamma,r,s)$ and cycle-shift actions.

\begin{lemma}\label{lem:cycle-shift}
		Let $G$ be a transitive permutation group on $\Omega$ and $|\Omega|>1$, and let
		$G\wr\mathbb{Z}_r$ act on $\Omega^s\times\mathbb{Z}_r$ via the cycle-shift action. Let $w\in\Omega$.
		This action is faithful. Moreover:
		\begin{itemize}
			\item [\rm (i)] the stabilizer of $\eta=(w,\ldots,w,\emptycommand{0})\in \Omega^s\times\mathbb{Z}_r$ is $G_w^s\times G^{r-s}$ if $s\le r$ and is $G_w^r$ if $s>r$;
			\item [\rm (ii)] $r\ge s$ if and only if $G\wr\mathbb{Z}_r$ is transitive;
			\item [\rm (iii)] if $s\mid r$, then $G$ has a regular subgroup if and only if $G\wr\mathbb{Z}_r$ has a regular subgroup.
		\end{itemize}
	\end{lemma}
\begin{proof}
	Let $W=G\wr\mathbb{Z}_r=G^r\rtimes\mathbb{Z}_r$, and write elements of $W$ as
	$(g_1,\ldots,g_r,\emptycommand{j})$.

	First we prove faithfulness.
	Assume $(g_1,\ldots,g_r,\emptycommand{j})\in W$ acts trivially on $\Omega^s\times\mathbb{Z}_r$.
	From the last coordinate we get $\emptycommand{j}=\emptycommand{0}$.
	Fix $t\in\{1,\ldots,r\}$, choose $i\equiv t-1\pmod r$, and act on
	$(w_1,\ldots,w_s,\emptycommand{i})$.
	The first coordinate becomes $w_1g_{i+1}=w_1g_t$, so $w_1g_t=w_1$ for all
	$w_1\in\Omega$, hence $g_t=1$.
	As $t$ is arbitrary, all $g_t=1$, so the action is faithful.

	An element $(g_1,\ldots,g_r,\emptycommand{j})$ fixes $\eta=(w,\ldots,w,\emptycommand{0})$ if and only if
	$\emptycommand{j}=\emptycommand{0}$ and $wg_t=w$ for every index $t$ appearing among
	$1,\ldots,s$ modulo $r$.
	If $s\le r$, these are exactly $t=1,\ldots,s$, so the stabilizer is
	$G_w^s\times G^{r-s}$.
	If $s>r$, every residue class modulo $r$ appears among $1,\ldots,s$, so each
	$g_t\in G_w$ and the stabilizer is $G_w^r$.
	This proves \textup{(i)}.

	For \textup{(ii)}, assume first $r\ge s$.
	Given $(u_1,\ldots,u_s,\emptycommand{i})$ and $(v_1,\ldots,v_s,\emptycommand{\ell})$ in
	$\Omega^s\times\mathbb{Z}_r$, choose $h_{i+t}\in G$ (indices modulo $r$) such that
	$u_th_{i+t}=v_t$ for $t=1,\ldots,s$. This is possible since $G$ is transitive.
	Choose the remaining coordinates of $h\in G^r$ arbitrarily.
	Then
	\[
		(u_1,\ldots,u_s,\emptycommand{i})(h_1,\ldots,h_r,\emptycommand{\ell-i})=(v_1,\ldots,v_s,\emptycommand{\ell}),
	\]
	so $W$ is transitive.
		Conversely, assume that $W$ is transitive.
		Then
		\[
			|W|/|W_{\eta}|=|\Omega^s\times\mathbb{Z}_r|=|\Omega|^s r.
		\]
		From \textup{(i)} we have
		$|W|/|W_{\eta}|\le |\Omega|^r r,$
		and hence $|\Omega|^s\le |\Omega|^r$.
        As $|\Omega|>1$, we have $s\le r$.
		Thus $W$ is transitive if and only if $r\ge s$.

	For \textup{(iii)}, assume that $s\mid r$.
	If $H\le G$ is regular on $\Omega$, let
	\[
		D=\{(h_1,\ldots,h_r)\in H^r\mid h_{i+s}=h_i\text{ for all }i\},
	\]
	where the subscripts are taken modulo $r$.
	Then $D\cong H^s$, and $D$ is normalized by the cyclic permutation of the
	$r$ coordinates. Thus $D\rtimes\mathbb{Z}_r\le G\wr\mathbb{Z}_r$ and acts regularly on
	$\Omega^s\times\mathbb{Z}_r$.

	Conversely, let $B\le G\wr\mathbb{Z}_r$ be regular on
	$\Omega^s\times\mathbb{Z}_r$. The projection of $B$ onto
	$\mathbb{Z}_r$ is transitive, and hence is all of $\mathbb{Z}_r$.
	Put $K=B\cap G^r$. Then
	\[
		|K|=\frac{|B|}{r}=|\Omega|^s.
	\]
	Moreover, $K$ acts semiregularly on the layer
	$\Omega^s\times\{0\}$, and hence regularly on this layer.
	Therefore $G^s$ contains a regular subgroup on
	$\Omega^s$. By \cite[Lemma 2.2]{dobson2023When}, this implies that $G$
	contains a regular subgroup.
\end{proof}

\subsection{Automorphism groups}

In this section, we compute $\Aut(\overrightarrow{\mathrm{PX}}(\Gamma,r,s))$ for $s\ge 2$.
First, we need some definitions and lemmas.

\begin{defi}
  Let $\Gamma$ be a digraph.
  Let $\Sigma=\overrightarrow{\mathrm{PX}}(\Gamma,r,s)$ where $s\ge 2$.
  Identify $V\Gamma\times(V\Gamma)^{s-1}$ with $(V\Gamma)^s$
  by $(x,(v_1,\ldots,v_{s-1}))=(x,v_1,\ldots,v_{s-1})$ and 
  identify $(V\Gamma)^{s-1}\times V\Gamma$ with $(V\Gamma)^s$ by
  $((v_1,\ldots,v_{s-1}),y)=(v_1,\ldots,v_{s-1},y)$.
  For $w=(v_1,\ldots,v_{s-1})\in (V\Gamma)^{s-1}$ and $\emptycommand{i}\in\mathbb{Z}_r$, define subsets
  $B_{w,\emptycommand{i}},C_{w,\emptycommand{i}}\subseteq V\Sigma$ by 
  \[
    B_{w,\emptycommand{i}}:=\{(x,w,\emptycommand{i})\in (V\Gamma)^s\times\mathbb{Z}_r\mid x\in V\Gamma\}
  \]
  and
  \[
    C_{w,\emptycommand{i}}:=\{(w,x,\emptycommand{i})\in (V\Gamma)^s\times\mathbb{Z}_r\mid x\in V\Gamma\}.
  \]
\end{defi}

\begin{lemma}\label{lem:loosely-attached}
  Let $\Gamma$ be a digraph, and assume that
  $\Gamma\times\overrightarrow{K}_2$ is connected.
  Let $s\ge 2$ and $\Sigma=\overrightarrow{\mathrm{PX}}(\Gamma,r,s)$.
  Then every alternet of $\Sigma$ has the form
  $\Sigma[B_{w,\emptycommand{i}}\cup C_{w,\emptycommand{i+1}}]=\Gamma\times\overrightarrow{K}_2$ for some
  $w\in (V\Gamma)^{s-1}$ and $\emptycommand{i}\in\mathbb{Z}_r$.
  In particular, $\Sigma$ is loosely attached.
\end{lemma}

\begin{proof}
  Let $\mathcal{A}$ be an alternet of $\Sigma$, and choose an arc
  \[
    (u_1,\ldots,u_s,\emptycommand{i})\to(v_1,\ldots,v_s,\emptycommand{i+1})
  \]
  in $\mathcal{A}$.
  Set $w=(u_2,\ldots,u_s)=(v_1,\ldots,v_{s-1})$.
  Then this arc lies in
  $\Sigma[B_{w,\emptycommand{i}}\cup C_{w,\emptycommand{i+1}}]$.
  By the definition of alternet (closure under sharing a tail or sharing a head),
  every arc in $\mathcal{A}$ must still have tail in $B_{w,\emptycommand{i}}$ and head in
  $C_{w,\emptycommand{i+1}}$.
  So $\mathcal{A}\subseteq \Sigma[B_{w,\emptycommand{i}}\cup C_{w,\emptycommand{i+1}}]$.
  Note that $\Sigma[B_{w,\emptycommand{i}}\cup C_{w,\emptycommand{i+1}}]=\Gamma\times\overrightarrow{K}_2$ is connected.
  Thus $\mathcal{A}=\Sigma[B_{w,\emptycommand{i}}\cup C_{w,\emptycommand{i+1}}]$ and every alternet has the required form.

  Finally, if two alternets are distinct, then their $B$-parts are distinct
  and their $C$-parts are also distinct, and they intersect in at most one
  vertex (a vertex is determined by its prefix and suffix of length $s-1$).
  Hence distinct alternets intersect in at most one vertex.
  Therefore $\Sigma$ is loosely attached.
\end{proof}

\begin{remark}\label{rmk:connected-alternets}
    Let $\Gamma$ be a digraph and $\Sigma$ be the graph whose vertex set is $V\Gamma$ and two vertices are adjacent if and only if they share a common out-neighbour in $\Gamma$.
    We point out that if $\Gamma\times\overrightarrow{K}_2$ is connected, then $\Gamma$ has no sinks and sources, and $\Sigma$ is connected.
    Conversely, if $\Gamma$ has no sinks and sources, and $\Sigma$ is connected, then $\Gamma\times\overrightarrow{K}_2$ is connected (see \cite{Chen2015Weak}).
\end{remark}

\begin{corollary}\label{cor:alternet-graph}
  Let $\Gamma$ be a digraph, let $m=|V\Gamma|$, and assume that
  $\Gamma\times\overrightarrow{K}_2$ is connected.
  Let $s\ge 2$ and $\Sigma=\overrightarrow{\mathrm{PX}}(\Gamma,r,s)$.
  Then $\mathsf{Al}(\Sigma)=\overrightarrow{\mathrm{PX}}(m,r,s-1)$.
\end{corollary}

\begin{proof}
  By the previous lemma, every alternet of $\Sigma$ has the form
  \[
    \mathcal{A}_{w,\emptycommand{i}}=\Sigma[B_{w,\emptycommand{i}}\cup C_{w,\emptycommand{i+1}}]
  \]
  for a unique pair $(w,\emptycommand{i})\in (V\Gamma)^{s-1}\times\mathbb{Z}_r$.
  Hence
  \[
    V\mathsf{Al}(\Sigma)=\{\mathcal{A}_{w,\emptycommand{i}}\mid (w,\emptycommand{i})\in (V\Gamma)^{s-1}\times\mathbb{Z}_r\}.
  \]

  Define
  \[
    \phi:V\mathsf{Al}(\Sigma)\to (V\Gamma)^{s-1}\times\mathbb{Z}_r,
    \qquad \phi(\mathcal{A}_{w,\emptycommand{i}})=(w,\emptycommand{i}).
  \]
  This is a bijection.

  Let $w=(w_1,\ldots,w_{s-1})$ and $u=(u_1,\ldots,u_{s-1})$.
  By construction of the sets $B_{w,\emptycommand{i}}$ and $C_{u,\emptycommand{i+1}}$,
  we have that if $s\ge 3$, then
  \[
    C_{w,\emptycommand{i+1}}\cap B_{u,\emptycommand{i+1}}\neq\emptyset
    \iff (u_1,\ldots,u_{s-2})=(w_2,\ldots,w_{s-1});
  \]
  if $s=2$, then $C_{w,\emptycommand{i+1}}\cap B_{u,\emptycommand{i+1}}\neq\emptyset$ for every $w,u\in V\Gamma$.
  Equivalently, in $\mathsf{Al}(\Sigma)$ there is an arc
  \[
    \mathcal{A}_{w,\emptycommand{i}}\to\mathcal{A}_{u,\emptycommand{i+1}}
  \]
  if and only if in $\overrightarrow{\mathrm{PX}}(m,r,s-1)$,
  \[
    \phi(\mathcal{A}_{w,\emptycommand{i}})
    \to \phi(\mathcal{A}_{u,\emptycommand{i+1}}).
  \]
  Therefore $\mathsf{Al}(\Sigma)= \overrightarrow{\mathrm{PX}}(m,r,s-1)$.
\end{proof}
	
Note that for $s=1$, $\overrightarrow{PX}(m,r,1)=\overrightarrow{C}_r[\overline{K_m}]$.
Hence $\Aut(\overrightarrow{\mathrm{PX}}(m,r,1))=\mathrm{S}_m\wr\mathbb{Z}_r$.
\begin{lemma}\label{lem:auto-of-De}
  Let $\Sigma=\overrightarrow{\operatorname{PX}}(m,r,s)$.
  Then $\Aut(\Sigma)=\mathrm{S}_m\wr \mathbb{Z}_r$. Moreover, if
  $\sigma=(\sigma_1,\ldots,\sigma_r,\emptycommand{j})\in\Aut(\Sigma)$ and
  $(v_1,\ldots,v_s,\emptycommand{i})\in V\Sigma$, then
  \[(v_1,\ldots,v_s,\emptycommand{i})\sigma=(v_1\sigma_{i+1},\ldots,v_s\sigma_{i+s},\emptycommand{i+j}).\]
\end{lemma}
\begin{proof}
  Note that $\Sigma=\overrightarrow{\operatorname{PX}}(\overrightarrow{B}(m,1),r,s)$.
  By Lemma~\ref{lem:loosely-attached},
  $\Sigma$ has complete bipartite alternets and is loosely attached.
  Moreover, $\Sigma$ is asymmetric and has no sinks or sources.
  By Proposition~\ref{prop:alternets} {\rm (iii)},
  $\mathsf{Pl}(\mathsf{Al}(\Sigma))=\Sigma$.
  Therefore, by Proposition~\ref{prop:alternets} {\rm (ii)},
  $\Aut(\mathsf{Al}(\Sigma))\le\Aut(\Sigma)$.
  On the other hand, by Proposition~\ref{prop:alternets} {\rm (i)},
  $\Aut(\Sigma)\le\Aut(\mathsf{Al}(\Sigma))$.
  Hence $\Aut(\Sigma)\cong\Aut(\mathsf{Al}(\Sigma))$.
  Furthermore, by Corollary~\ref{cor:alternet-graph},
  \[
    \mathsf{Al}(\Sigma)=\overrightarrow{\mathrm{PX}}(m,r,s-1),
  \]
  and thus
  \[
    \Aut(\Sigma)\cong\Aut(\overrightarrow{\mathrm{PX}}(m,r,s-1))
    \cong\cdots\cong \Aut(\overrightarrow{\mathrm{PX}}(m,r,1))
    \cong \mathrm{S}_m\wr\mathbb{Z}_r.
  \]
  Conversely, the action of $\mathrm{S}_m\wr\mathbb{Z}_r$ on $V\Sigma$ defined above is the cycle-shift action, hence well-defined.
  One can also check that $\mathrm{S}_m\wr\mathbb{Z}_r$ consists of
  digraph automorphisms. This completes the proof.
\end{proof}

This gives an alternative proof of the fact that the automorphism group of a directed Praeger--Xu graph $\overrightarrow{\mathrm{PX}}(m,r,s)$ is $\mathrm{S}_m\wr\mathbb{Z}_r$ (see \cite[Theorem 2.8]{Praeger1989high}).


\begin{theorem}\label{thm:aut-general-px}
	Let $\Gamma$ be a digraph such that $\Gamma\times\overrightarrow{K}_2$ is connected, and let $\Sigma=\overrightarrow{\mathrm{PX}}(\Gamma,r,s)$ where $s\ge 1$.
	Let $W=\mathrm{Sym}(V\Gamma)\wr\mathbb{Z}_r$.
  Then \[\Aut(\Sigma)=\{(g_1,\ldots,g_r,\emptycommand{k})\in W\mid \forall i, g_i\to g_{i+s}\}\le W,\] and $W$ acts on
  $V\Sigma$ by the cycle-shift action.
  In particular, $\Aut(\Gamma[r])=\Aut(\Gamma)\wr\mathbb{Z}_r$.
\end{theorem}
\begin{proof}
    If $s=1$, then $\Sigma=\Gamma\times\overrightarrow{C}_r$.
    Let $B_i=\{(v,i)\in V\Sigma\mid v\in V\Gamma\}$ for $i\in \mathbb{Z}_r$.
    Define $\hat\Sigma$ be a graph whose vertices are the same to $\Sigma$ and two vertices are adjacent if and only if they share an out-neighbour in $\Sigma$.
    Therefore, two vertices are adjacent only if they lies in the same $B_i$.
    Since $\Gamma[B_i\cup B_{i+1}]=\Gamma\times\overrightarrow{K}_2$ is connected, we have $\hat\Sigma[B_i]$ is also connected.
    So $\hat\Sigma[B_i]$ forms a connected component of $\hat\Sigma$.
    As $\Aut(\Sigma)$ preserves the adjacency of $\hat\Sigma$, $\Aut(\Sigma)$ preserves connected components of $\hat\Sigma$, which means that $\Aut(\Sigma)$ preserves \[\{B_i\mid i\in\mathbb{Z}_r\}.\]
    Let $\sigma\in\Aut(\Sigma)$.
    Then there exist $g_i\in\Sym(B_i)$ and $\eta\in\Sym(\mathbb{Z}_r)$ such that for every $(v,j)\in V\Sigma$,
    \[(v,j)\sigma=(vg_j,j\eta)\tag{$\ast$}\label{eq:s=1_case_1}\]
    By the definition of $\Sigma$, one can check that 
    \[g_i\to g_{i+1}, \forall i\in\mathbb{Z}_r\quad \text{and }\quad \eta\in\Aut(C_r)\tag{$\ast\ast$}\label{eq:s=1_case_2}\]
    Conversely, one can also check that any permutation satisfying (\ref{eq:s=1_case_1}) and (\ref{eq:s=1_case_2}) is an automorphism of $\Sigma$.
    This completes the case where $s=1$.

    Now suppose $s\ge 2$. 
	Let $m=|V\Gamma|$, and put $S=\Aut(\Sigma)$.
	By Lemma~\ref{lem:loosely-attached}, $\Sigma$ is loosely attached.
	By Corollary~\ref{cor:alternet-graph},
	\[
		\mathsf{Al}(\Sigma)=\overrightarrow{\mathrm{PX}}(m,r,s-1).
	\]

	Let $W$ act on $V\Sigma$ by the cycle-shift action.
	This action preserves alternets, so it induces a faithful permutation group
	$\emptycommand{W}$ on $V\mathsf{Al}(\Sigma)$.
	Using Corollary~\ref{cor:alternet-graph} and Lemma~\ref{lem:auto-of-De}, we have
	\[
		\emptycommand{W}=\Aut(\mathsf{Al}(\Sigma)).
	\]

	Since $\Sigma$ is asymmetric and has no sinks or sources, by
	Proposition~\ref{prop:alternets} {\rm (ii)}, the induced action
	$\emptycommand{S}$ of $S$ on $\mathsf{Al}(\Sigma)$ satisfies $\emptycommand{S}\le\emptycommand{W}$.
	Hence $S\le W$.

	Now set
	$T:=\{(g_1,\ldots,g_r,\emptycommand{k})\in W\mid \forall i,\ g_i\to g_{i+s}\}.$
	We show that $S=T$.
	Let $g=(g_1,\ldots,g_r,\emptycommand{i})\in S\le W$.
	For any arc $(u,v)\in A\Gamma$ and any $i$, consider the arc of $\Sigma$
	\[
		(u,w,\emptycommand{i-1})\to(w,v,\emptycommand{i}),
	\]
	where $w\in (V\Gamma)^{s-1}$.
	Applying $g$ and using that $g$ preserves arcs of $\Sigma$, we obtain
	$(ug_{i},vg_{i+s})\in A\Gamma$.
	Thus $g_i\to g_{i+s}$ for all $i$, so $g\in T$.
	Hence $S\subseteq T$.
	Conversely, let $g=(g_1,\ldots,g_r,\emptycommand{k})\in T$.
	Take any arc
	\[
		(x_1,\ldots,x_s,\emptycommand{i})\to(y_1,\ldots,y_s,\emptycommand{i+1})
	\]
	of $\Sigma$.
	Then $(x_1,y_s)\in A\Gamma$ and $y_j=x_{j+1}$ for $1\le j\le s-1$.
	By $g_{i+1}\to g_{i+s+1}$,
	\[
		(x_1g_{i+1},y_sg_{i+s+1})\in A\Gamma,
	\]
	and clearly $y_jg_{j+i+1}=x_{j+1}g_{j+i+1}$ for $1\le j\le s-1$.
	So the images of the two endpoints under the cycle-shift action are again adjacent,
	and therefore $g$ preserves arcs.
	Hence $g\in S$, so $T\subseteq S$.
	Therefore $S=T$.
	
	If $r=s$, then $i\equiv i+s\pmod r$.
	Hence $g_i\to g_i$, which is equivalent to $g_i\in\Aut(\Gamma)$.
	Therefore $\Aut(\Gamma[r])=\Aut(\Gamma)\wr\mathbb{Z}_r$.
	This completes the proof.
\end{proof}

\begin{remark}\label{rmk:necessarity-connected-alternets}
    The connectivity of $\Gamma\times\overrightarrow{K}_2$ is crucial in Theorem~\ref{thm:aut-general-px}.
    For example, let $\Gamma=K_2$. 
    Then $\Gamma^{*3}=\overrightarrow{C}_6\cup K_2$ and $\Gamma[3]=(\overrightarrow{C}_6\cup K_2)\times\overrightarrow{C}_3=4\overrightarrow{C}_6$. 
    Thus, $\Aut(\Gamma)\wr \mathbb{Z}_3\cong\mathbb{Z}_2\wr\mathbb{Z}_3$ while $\Aut(\Gamma[3])\cong\mathbb{Z}_6\wr\S_4$.
\end{remark}

\subsection{Connectivity and transitivity}
With the results established in the previous sections, we now examine the connectivity and
transitivity of $\overrightarrow{\mathrm{PX}}(\Gamma,r,s)$.


\begin{theorem}\label{cor:nice-connectivity}
    Let $\Gamma$ be a digraph, and assume that $\Gamma\times\overrightarrow{K}_2$ is connected.
    Then $\overrightarrow{\operatorname{PX}}(\Gamma,r,s)$ is connected.
\end{theorem}
\begin{proof}
    Put $\Sigma=\overrightarrow{\mathrm{PX}}(\Gamma,r,s)$.
    We first claim that if $\mathsf{Al}(\Sigma)$ is connected, then $\Sigma$ is connected.
    Note that every alternet is connected.
    Then by the connectivity of $\mathsf{Al}(\Sigma)$, $\Sigma$ is connected.

    Note that $\overrightarrow{\mathrm{PX}}(m,r,1)=\overrightarrow{C}_r[\overline{\emptycommand{K}}_m]$.
    Hence $\overrightarrow{\mathrm{PX}}(m,r,1)$ is connected.
    Assume $s\ge 2$.
    By Corollary~\ref{cor:alternet-graph}, $\mathsf{Al}(\Sigma)=\overrightarrow{\mathrm{PX}}(m,r,s-1)$.
    It is sufficient to prove that $\overrightarrow{\mathrm{PX}}(m,r,s-1)$ is connected.
    Apply Corollary~\ref{cor:alternet-graph} recursively.
    We obtain that it is sufficient to prove that $\overrightarrow{\mathrm{PX}}(m,r,1)$ is connected.
    This completes the proof.
\end{proof}

\begin{defi}
	Let $\Gamma$ be a digraph, and let $\ell\ge 1$ and $t\ge 0$.
	Let \[\mathcal{P}_\ell(\Gamma)=\{(\sigma_1,\ldots,\sigma_\ell)\in\operatorname{Sym}(V\Gamma)^\ell\mid \forall i, \sigma_i\to \sigma_{i+1}\},\]
	the set of closed directed walks of length $\ell$ in graph $(\operatorname{Sym}(V\Gamma),\Aut^{\operatorname{TF}}(\Gamma))$, where indices of $\sigma$ are taken modulo $\ell$.
	One can check that $\mathcal{P}_\ell(\Gamma)$ is a group.
    
    Let \[W_t(\Gamma)=\{(v_0,\ldots,v_t)\in (V\Gamma)^{t+1}\mid \forall 1\le i\le t,v_{i-1}\to v_i\}\] be the set of directed walks of length $t$ in $\Gamma$. In particular, $W_0(\Gamma)=V\Gamma$.
    The group $\mathcal{P}_{\ell}(\Gamma)$ has an action on $W_t(\Gamma)$ given by
    \[(v_0,\ldots,v_t)(\sigma_1,\ldots,\sigma_\ell)=(v_0\sigma_1,v_1\sigma_2,\ldots,v_t\sigma_{t+1})\]
    where indices of $\sigma$ are taken modulo $\ell$.
    We call this action the \emph{closed walk action}.
\end{defi}

Now we are ready to prove the main result of this section, which characterizes the vertex-transitivity and arc-transitivity of $\overrightarrow{\mathrm{PX}}(\Gamma,r,s)$.
To clarify the main idea, we address the specific case where $s\mid r$ (Proposition~\ref{thm:ver-trans} and Proposition~\ref{thm:arc-transitivity}) first, and then extend the same approach to the general case (Theorem~\ref{thm:arc-transitivity-new}).

Note that if $\Gamma$ is asymmetric, $W_t(\Gamma)$ is the set of $t$-arcs of $\Gamma$ for $t\ge 1$.
However, if $\Gamma$ is a graph, $W_t(\Gamma)$ is the set of $t$-walks of $\Gamma$ for $t\ge 1.$

\begin{proposition}\label{thm:ver-trans}
    Let $\Gamma$ be a digraph such that $\Gamma\times\overrightarrow{K}_2$ is connected and let $r\ge 3$, $s\ge 1$ and $t\ge 1$ with $s\mid r$.
    Write $\ell=r/s$ and let $\Sigma=\overrightarrow{\operatorname{PX}}(\Gamma,r,s)$.
    Then $\Sigma$ is vertex-transitive if and only if the closed walk action of $\mathcal{P}_\ell(\Gamma)$ on $W_0(\Gamma)$ is transitive.
    In particular, if $\Gamma$ is vertex-transitive, then $\Sigma$ is vertex-transitive.
\end{proposition}
\begin{proof}
    Let $B_0=\{(x_1,\ldots,x_s,0)\in V\Sigma\mid x_i\in V\Gamma\}$.
    For every vertex $(y_1,\ldots,y_s,i)\in V\Sigma$, the automorphism $(\mathrm{id},\ldots,\mathrm{id},-i)$ sends it to a vertex in $B_0$.
    Moreover, by Theorem~\ref{thm:aut-general-px}, an automorphism taking one element of $B_0$ to another has the form $(g_1,\ldots,g_r,0)$.
    Therefore, $\Sigma$ is vertex-transitive if and only if $A_0$ acts on $B_0$ transitively.
    Let $\psi:B_0\to W_0(\Gamma)^s,(x_1,\ldots,x_s,0)\mapsto (x_1,\ldots,x_s)$. Then $\psi$ is bijective.

    Let $\mathcal{P}_\ell(\Gamma)^s$ act on $W_0(\Gamma)^s$ by the product action.
    For $g=(g_1,\ldots,g_r,0)\in A_0$, define
    \[
        \varphi(g)=\bigl((g_1,g_{1+s},\ldots,g_{1+(\ell-1)s}),\ldots,
        (g_s,g_{2s},\ldots,g_{s+(\ell-1)s})\bigr).
    \]
    By Theorem~\ref{thm:aut-general-px}, $\varphi:A_0\to\mathcal{P}_\ell(\Gamma)^s$ is a well-defined group isomorphism. Moreover,
    \[
        \psi(lg)=\psi(l)\varphi(g)
        \qquad\text{for all }l\in L\text{ and }g\in A_0.
    \]
    Hence $\Sigma$ is vertex-transitive if and only if $\mathcal{P}_\ell(\Gamma)^s$ is transitive on $W_0(\Gamma)^s$.
    The latter is equivalent to that $\mathcal{P}_\ell(\Gamma)$ is transitive on $W_0(\Gamma)$.

    Now suppose that $\Gamma$ is vertex-transitive.
    Let $P=\{(\alpha,\ldots,\alpha)\in\Sym(V\Gamma)^\ell\mid \alpha\in\Aut(\Gamma)\}$.
    Then $P\le\mathcal{P}_\ell(\Gamma)$ and $P$ is transitive on $W_0(\Gamma)$, which implies that $\mathcal{P}_\ell(\Gamma)$ is transitive on $W_0(\Gamma)$. 
    This completes the proof.
\end{proof}

\begin{proposition}
    \label{thm:arc-transitivity}
    Let $\Gamma$ be a digraph such that $\Gamma\times\overrightarrow{K}_2$ is connected and let $r\ge 3$, $s\ge 1$ and $t\ge 1$ with $s\mid r$. Let $\ell=r/s$, $q=\lceil t/s\rceil$ and $\Sigma=\overrightarrow{\operatorname{PX}}(\Gamma,r,s)$.
    Then $\Sigma$ is $t$-arc-transitive if and only if the closed walk action of $\mathcal{P}_\ell(\Gamma)$ on $W_{q}(\Gamma)$ is transitive.
\end{proposition}
\begin{proof}
    As $\Sigma$ is asymmetric, $W_t(\Sigma)$ is the set of $t$-arcs of $\Sigma$.
    Let $B_0=\{(x_1,\ldots,x_s,0)\in V\Sigma\mid x_i\in V\Gamma\}$ and let
    \[
        L=\{(v_0,\ldots,v_t)\in W_t(\Sigma)\mid v_0\in B_0\}
    \]
    be the set of $t$-arcs whose starting vertex lies in $B_0$.
    The automorphism $({\rm id},\ldots,{\rm id},-i)$ sends every vertex in layer $i$ to $B_0$.
    Moreover, by Theorem~\ref{thm:aut-general-px}, an automorphism taking one element of $L$ to another has the form $(g_1,\ldots,g_r,0)$.
    Let
    \[
        A_0=\{(g_1,\ldots,g_r,0)\in\Aut(\Sigma)\}.
    \]
    Therefore, $\Sigma$ is $t$-arc-transitive if and only if $A_0$ acts on $L$ transitively.

    Let $l=(v_0,\ldots,v_t)\in L$.
    By the definition of $\Sigma$, there exist unique $s+t$ vertices $x_1,\ldots,x_{s+t}$ of $\Gamma$ such that $x_i\to x_{i+s}$ for $1\le i\le t$ and
    \[\begin{aligned}
        v_0&=(x_1,\ldots,x_s,0),\\
        v_1&=(x_2,\ldots,x_s,x_{s+1},1),\\
        &\ \vdots\\
        v_t&=(x_{t+1},\ldots,x_{t+s},t).
    \end{aligned}\]
    We represent $l$ by $\underline{x_1,\ldots,x_{s+t}}$. Then the action of $A_0$ on $L$ is given by
    \[
        \underline{x_1,\ldots,x_{s+t}}(g_1,\ldots,g_r,0)
        =\underline{x_1g_1,\ldots,x_{s+t}g_{s+t}},
    \]
    where indices of $g$ are taken modulo $r$.

    Let $e=t-(q-1)s$. For $1\le i\le s$, let
    \[
        k_i=\left\lfloor\frac{s+t-i}{s}\right\rfloor
        =\begin{cases}q,&i\le e,\\q-1,&i>e,\end{cases}
        \qquad\text{and}\qquad
        w_i=(x_i,x_{i+s},\ldots,x_{i+k_is}).
    \]
    Then $w_i\in W_{k_i}(\Gamma)$. Thus
    \[
        \psi:L\longrightarrow W_q(\Gamma)^e\times W_{q-1}(\Gamma)^{s-e},
        \qquad
        \underline{x_1,\ldots,x_{s+t}}\longmapsto(w_1,\ldots,w_s)
    \]
    is a bijection.

    Let $\mathcal{P}_\ell(\Gamma)^s$ act on $W_q(\Gamma)^e\times W_{q-1}(\Gamma)^{s-e}$ by the product action, where each factor acts by the closed walk action. For $g=(g_1,\ldots,g_r,0)\in A_0$, define
    \[
        \varphi(g)=\bigl((g_1,g_{1+s},\ldots,g_{1+(\ell-1)s}),\ldots,
        (g_s,g_{2s},\ldots,g_{s+(\ell-1)s})\bigr).
    \]
    By Theorem~\ref{thm:aut-general-px}, $\varphi:A_0\to\mathcal{P}_\ell(\Gamma)^s$ is a well-defined group isomorphism. Moreover,
    \[
        \psi(lg)=\psi(l)\varphi(g)
        \qquad\text{for all }l\in L\text{ and }g\in A_0.
    \]
    Hence $\Sigma$ is $t$-arc-transitive if and only if $\mathcal{P}_\ell(\Gamma)^s$ is transitive on $W_q(\Gamma)^e\times W_{q-1}(\Gamma)^{s-e}$.

    If the closed walk action of $\mathcal{P}_\ell(\Gamma)$ on $W_q(\Gamma)$ is transitive, then its action on $W_{q-1}(\Gamma)$ is also transitive because $\Gamma$ has no sources or sinks (Remark~\ref{rmk:connected-alternets}). Hence the product action is transitive. Conversely, since $e\ge1$, transitivity of the product action implies transitivity on its first $W_q(\Gamma)$ factor. This proves the result.
\end{proof}

\begin{corollary}
    \label{thm:arc-transitivity-apply}
    Let $\Gamma$ be a digraph such that $\Gamma\times\overrightarrow{K}_2$ is connected, and let $s\ge 1$ and $r\ge 3$ be integers such that $s\mid r$. Let $\Sigma=\overrightarrow{\operatorname{PX}}(\Gamma,r,s)$.
    \begin{itemize}
        \item [\rm (i)] If $\Gamma$ is arc-transitive, then $\Sigma$ is $s$-arc-transitive. Moreover, $\Gamma[r]$ is $r$-arc-transitive if and only if $\Gamma$ is arc-transitive.
        \item [\rm (ii)] If $\Gamma$ is asymmetric and $k$-arc-transitive, then $\Sigma$ is $ks$-arc-transitive. 
        \item [\rm (iii)] If $\Gamma$ is asymmetric, then $\Gamma[r]$ is $t$-arc-transitive if and only if $\Gamma$ is $\lceil t/r\rceil$-arc-transitive.
    \end{itemize}
\end{corollary}
\begin{proof}
    Write $\ell=r/s$.
    
    {\rm (i)} Let $\Gamma$ be arc-transitive and let \[A_\ell=\{\underbrace{(\alpha,\ldots,\alpha)}_\ell\mid\alpha\in\Aut(\Gamma)\}.\]
    Then $A_\ell\le \mathcal{P}_\ell(\Gamma)$.
    Note that $A_\ell$ acts transitively on $W_1(\Gamma)$ in the closed walk action as $\Gamma$ is arc-transitive.
    Then by Theorem~\ref{thm:arc-transitivity}, $\Sigma$ is $s$-arc-transitive.
    
    Now suppose that $r=s$ and $\Sigma$ is $s$-arc-transitive. Then $\ell=1$.
    By Theorem~\ref{thm:arc-transitivity}, $\Aut(\Gamma)=\mathcal{P}_\ell(\Gamma)$ acts transitively on $W_1(\Gamma)$, which is the set of arcs of $\Gamma$.
    Hence $\Gamma$ is arc-transitive.

    {\rm (ii)} If $\Gamma$ is asymmetric, then $W_t(\Gamma)$ is the set of $t$-arc of $\Gamma$ for every $t\ge 1$.
    Let \[A_\ell=\{\underbrace{(\alpha,\ldots,\alpha)}_\ell\mid\alpha\in\Aut(\Gamma)\}.\]
    Then $A_\ell\le \mathcal{P}_\ell(\Gamma)$ and $A_\ell$ acts transitively on $W_k(\Gamma)$ in the closed walk action as $\Gamma$ is $k$-arc-transitive. 
    By Theorem~\ref{thm:arc-transitivity}, $\Sigma$ is $ks$-arc-transitive.

    {\rm (iii)} By {\rm (ii)}, we only need to prove the necessity.
    Let $r=s$ and $\Gamma[r]$ be $t$-arc-transitive.
    Then $\ell=1$.
    By Theorem~\ref{thm:arc-transitivity}, $\Aut(\Gamma)=\mathcal{P}_\ell(\Gamma)$ acts transitively on $W_{\lceil t/r\rceil}(\Gamma)$, which is the set of $\lceil t/r\rceil$-arcs of $\Gamma$. Hence $\Gamma$ is $\lceil t/r\rceil$-arc-transitive.
\end{proof}

To remove the restriction $s\mid r$, we need to consider several closed walk actions of the same group at the same time.

\begin{defi}\label{def:twisted-closed-walk}
    Let $\Gamma$ be a digraph, let $r,s\ge 1$, $d=\gcd(r,s)$, $\ell=r/d$ and $m=s/d$.
    Let $\sigma=(\sigma_1,\ldots,\sigma_\ell)\in \mathcal{P}_\ell(\Gamma)$.
    Define \[\sigma^{(n)}=(\sigma_{1+n},\ldots,\sigma_{\ell+n}).\]
    Then $\sigma^{(n)}\in\mathcal{P}_\ell(\Gamma)$.
    Choose an integer $c$ such that $cs\equiv d\pmod r$. Such an integer exists because $\gcd(m,\ell)=1$, and it is unique modulo $\ell$.
    For non-negative integers $k_0,\ldots,k_{m-1}$, define an action of $\mathcal{P}_\ell(\Gamma)$ on
    \[\prod_{j=0}^{m-1}W_{k_j}(\Gamma)\]
    by
    \[
        (w_0,\ldots,w_{m-1})\sigma
        =\bigl(w_0\sigma,w_1\sigma^{(c)},\ldots,w_{m-1}\sigma^{((m-1)c)}\bigr),
    \]
    and each factor is acted on by the closed walk action. All indices of $\sigma$ are taken modulo $\ell$.
    We call this the \emph{$+d$-twisted closed walk action} for $(r,s)$.
\end{defi}

One can check that this action is well-defined.
More explicitly, if $w_j=(v_{j,0},\ldots,v_{j,k_j})$, then its image in the $j$-th factor is
\[
    (v_{j,0}\sigma_{1+jc},v_{j,1}\sigma_{2+jc},\ldots,v_{j,k_j}\sigma_{k_j+1+jc}).
\]

For convenience, in the following theorem, a $0$-arc-transitive digraph means that this digraph is vertex-transitive.

\begin{theorem}\label{thm:arc-transitivity-new}
    Let $\Gamma$ be a digraph such that $\Gamma\times\overrightarrow{K}_2$ is connected, and let $r\ge 3$, $s\ge 1$ and $t\ge 0$.
    Let $d=\gcd(r,s)$, $\ell=r/d$, $m=s/d$ and $\Sigma=\overrightarrow{\operatorname{PX}}(\Gamma,r,s)$, and let
    \[
        q=\max\{1,\lceil t/s\rceil\},\qquad
        e=t-(q-1)s\quad\text{ and }\quad h=\lceil e/d\rceil.
    \]
    Then $0\le e\le s$ and $0\le h\le m$, and $\Sigma$ is $t$-arc-transitive if and only if $\mathcal{P}_\ell(\Gamma)$ is transitive on
    \[
        W_q(\Gamma)^h\times W_{q-1}(\Gamma)^{m-h}
    \]
    in the $+d$-twisted closed walk action for $(r,s)$.
    In particular, when $t=0$, we have $q=1$ and $e=h=0$, so $\Sigma$ is vertex-transitive if and only if $\mathcal{P}_\ell(\Gamma)$ is transitive on $W_0(\Gamma)^m=(V\Gamma)^m$ in this action.
\end{theorem}
\begin{proof}
    We regard a $0$-arc as a single vertex. Since $\Sigma$ is asymmetric, $W_t(\Sigma)$ is the set of $t$-arcs of $\Sigma$ for every $t\ge 0$.
    Let $B_0=\{(x_1,\ldots,x_s,0)\in V\Sigma\mid x_i\in V\Gamma\}$
    and let
    \[L=\{(v_0,\ldots,v_t)\in W_t(\Sigma)\mid v_0\in B_0\}\]
    be the set of $t$-arcs whose starting vertex lies in $B_0$.
    The automorphism $({\rm id},\ldots,{\rm id},-i)$ sends every vertex in layer $i$ to $B_0$.
    Moreover, by Theorem~\ref{thm:aut-general-px}, an automorphism sending one element of $L$ to another has the form $(g_1,\ldots,g_r,0)$.
    Let
    \[A_0=\{(g_1,\ldots,g_r,0)\in\Aut(\Sigma)\}.
    \]
    Therefore, $\Sigma$ is $t$-arc-transitive if and only if $A_0$ acts on $L$ transitively.

    Let $l=(v_0,\ldots,v_t)\in L$.
    By the definition of $\Sigma$, there exist unique $s+t$ vertices $x_1,\ldots,x_{s+t}$ of $\Gamma$ such that
    \[
        x_i\to x_{i+s}\quad(1\le i\le t),\qquad
        v_j=(x_{j+1},\ldots,x_{j+s},j)\quad(0\le j\le t),
    \]
    where the layer coordinate is read modulo $r$. For $t=0$, the arc conditions are vacuous, and this simply represents a vertex of $B_0$ by its $s$ entries.
    Therefore, we may represent $l$ as $\underline{x_1,\ldots,x_{s+t}}$, and the action of $A_0$ on $L$ is represented by
    \[
        \underline{x_1,\ldots,x_{s+t}}(g_1,\ldots,g_r,0)
        =\underline{x_1g_1,\ldots,x_{s+t}g_{s+t}},
    \]
    where indices of $g$ are taken modulo $r$.

    For $1\le i\le s$, let
    \[
        k_i=\left\lfloor\frac{s+t-i}{s}\right\rfloor
        =\begin{cases}q,&i\le e,\\q-1,&i>e,\end{cases}
    \]
    and
    \[
    w_i=(x_i,x_{i+s},\ldots,x_{i+k_is}).
    \]
    Then $w_i$ is a directed walk in $\Gamma$ of length $k_i$.
    Now we rewrite $\underline{x_1,\ldots,x_{s+t}}$ by these walks, grouping together those whose starting indices are congruent modulo $d$.
    This gives a bijection
    \[
        \psi:L\longrightarrow\prod_{a=1}^d X_a,\qquad
        \psi(l)=\bigl((w_a,w_{a+d},\ldots,w_{a+(m-1)d})\bigr)_{a=1}^d,
    \]
    where
    \[X_a=\prod_{j=0}^{m-1}W_{k_{a+jd}}(\Gamma).\]
    Indeed, the walks $w_i$ partition the $s+t$ entries $x_1,\ldots,x_{s+t}$, and the only restrictions on these entries are the arcs within each walk.

    For $g=(g_1,\ldots,g_r,0)\in A_0$, define
    \[
        \rho^{(a)}=(g_a,g_{a+s},\ldots,g_{a+(\ell-1)s})
        \qquad(1\le a\le d).
    \]
    Addition by $s$ on $\mathbb{Z}_r$ has $d$ cycles of length $\ell$, represented by $1,\ldots,d$. By Theorem~\ref{thm:aut-general-px}, the conditions $g_i\to g_{i+s}$ say exactly that each $\rho^{(a)}$ belongs to $\mathcal{P}_\ell(\Gamma)$. Thus
    \[
        \varphi:A_0\longrightarrow\mathcal{P}_\ell(\Gamma)^d,\qquad
        g\longmapsto(\rho^{(a)})_{a=1}^d
    \]
    is a group isomorphism: it rearranges the $r$ coordinates, and its inverse assigns the entries of the $d$ closed walks to their respective cycles.

    Choose $c$ as in Definition~\ref{def:twisted-closed-walk}, so that $cs\equiv d\pmod r$.
    For $0\le j\le m-1$ and $0\le u\le k_{a+jd}$, we have
    \[
        g_{a+jd+us}=\rho^{(a)}_{1+jc+u}.
    \]
    Consequently, the entries of $w_{a+jd}$ are acted on by the closed walk action of $(\rho^{(a)})^{(jc)}$, so $\rho^{(a)}$ acts on $X_a$ by the $+d$-twisted closed walk action.
    Let $\mathcal{P}_\ell(\Gamma)^d$ act on $\prod_{a=1}^d X_a$ by these actions in the respective factors. Then
    \[
        \psi(lg)=\psi(l)\varphi(g)
        \qquad\text{for all }l\in L\text{ and }g\in A_0.
    \]
    Therefore, $\Sigma$ is $t$-arc-transitive if and only if $\mathcal{P}_\ell(\Gamma)$ is transitive on every $X_a$ in the twisted action.

    For $1\le a\le d$, let
    \[
        h_a=\bigl|\{j\in\{0,\ldots,m-1\}:a+jd\le e\}\bigr|.
    \]
    The factors of $X_a$ with walk length $q$ are precisely its first $h_a$ factors. Hence
    \[
        X_a=W_q(\Gamma)^{h_a}\times W_{q-1}(\Gamma)^{m-h_a},
        \qquad h_1=\lceil e/d\rceil=h,\qquad 0\le h_a\le h.
    \]
    In particular, $X_1$ is the set in the statement. There is a map $X_1\to X_a$ which deletes the final vertex in each factor indexed by $h_a\le j<h$ and leaves the other factors unchanged. 
    The map is commutative with the action of $\mathcal{P}_\ell(\Gamma)$ because truncation commutes with the closed walk action. It is surjective because $q\ge1$ and every directed walk of length $q-1$ can be extended to a directed walk of length $q$, as $\Gamma$ has no sinks (Remark~\ref{rmk:connected-alternets}). Thus transitivity on $X_1$ implies transitivity on every $X_a$. 
    The converse follows by taking $a=1$.

    When $t=0$, all $k_i$ and $h_a$ are zero, so no truncation is needed and every $X_a$ is $W_0(\Gamma)^m$. Explicitly, the twisted action on this set is
    \[
        (y_0,\ldots,y_{m-1})(\sigma_1,\ldots,\sigma_\ell)
        =(y_0\sigma_1,y_1\sigma_{1+c},\ldots,y_{m-1}\sigma_{1+(m-1)c}).
    \]
    The same action isomorphism therefore gives exactly the claimed vertex-transitivity criterion.
\end{proof}

\begin{remark}
    If $s\mid r$, then $d=s$ and $m=1$. The twisted action has only one factor and is the closed walk action. For $t=0$, we have $h=0$ and the action is on $W_0(\Gamma)$, recovering Proposition~\ref{thm:ver-trans}. For $t\ge1$, we have $h=1$ and the action is on $W_{\lceil t/s\rceil}(\Gamma)$, recovering Proposition~\ref{thm:arc-transitivity}.
\end{remark}

Theorem~\ref{thm:arc-transitivity-new} also gives a restriction on the possible degrees of arc-transitivity without assuming $s\mid r$.

\begin{corollary}\label{cor:arc-transitivity-blocks}
    Let $\Gamma\times\overrightarrow{K}_2$ be connected, let $r\ge3$, $s\ge1$, and put $d=\gcd(r,s)$ and $\Sigma=\overrightarrow{\operatorname{PX}}(\Gamma,r,s)$. For $t\ge0$, let
    \[
        q=\max\{1,\lceil t/s\rceil\},\qquad e=t-(q-1)s,\qquad
        t'=d\lceil t/d\rceil.
    \]
    Then $\Sigma$ is $t$-arc-transitive if and only if it is $t'$-arc-transitive. In particular, if $\Sigma$ is vertex-transitive, then the largest degree of arc-transitivity of $\Sigma$ is divisible by $d$.
\end{corollary}
\begin{proof}
    Since $d\mid s$, we have $t\le t'\le qs$. If $t=0$, then $q=1$, $e=0$ and $t'=0$, so the equivalence is immediate. If $t\ge1$, then $1\le e\le s$ and $(q-1)s<t'\le qs$. Thus the integers $t$ and $t'$ give the same values of $q$ in Theorem~\ref{thm:arc-transitivity-new}.
    Let $e'=t'-(q-1)s$.
    Then $e'/d=\lceil t/d\rceil-(q-1)s/d=\lceil e/d\rceil$, which implies that  $\lceil e'/d\rceil=\lceil e/d\rceil=h$ in Theorem~\ref{thm:arc-transitivity-new}, so the two transitivity conditions are identical.

    If $T\ge0$ is the largest degree of arc-transitivity of $\Sigma$, the equivalence gives $d\lceil T/d\rceil$-arc-transitivity. Since $d\lceil T/d\rceil\ge T$, maximality forces equality, and hence $d\mid T$. This includes the case $T=0$.
\end{proof}

Here is an exact criterion when $\Gamma$ has no non-diagonal two-fold automorphisms.

\begin{corollary}\label{cor:arc-transitivity-diagonal}
    Let $\Gamma$ have at least two vertices, let $\Gamma\times\overrightarrow{K}_2$ be connected, and suppose that every two-fold automorphism of $\Gamma$ has the form $(\alpha,\alpha)$ with $\alpha\in\Aut(\Gamma)$. Let $r\ge3$, $s\ge1$, $t\ge0$ and $\Sigma=\overrightarrow{\operatorname{PX}}(\Gamma,r,s)$. Then $\Sigma$ is $t$-arc-transitive if and only if $s\mid r$ and $\Aut(\Gamma)$ is transitive on $W_{\lceil t/s\rceil}(\Gamma)$. In particular, $\Sigma$ is vertex-transitive if and only if $s\mid r$ and $\Gamma$ is vertex-transitive.
\end{corollary}
\begin{proof}
    Under the hypothesis, $\mathcal{P}_\ell(\Gamma)$ consists of the constant tuples $(\alpha,\ldots,\alpha)$ with $\alpha\in\Aut(\Gamma)$, for every $\ell\ge1$. If $s\mid r$, the twisted action in Theorem~\ref{thm:arc-transitivity-new} has unique factor, which is $W_{\lceil t/s\rceil}(\Gamma)$, and it is $W_0(\Gamma)$ when $t=0$. 
    The action on this factor is precisely the natural action of $\Aut(\Gamma)$, proving the equivalence in this case.

    Let $d,\ell,m,q,e,h$ be defined as in Theorem~\ref{thm:arc-transitivity-new}.
    Suppose $s\nmid r$. Then $m\ge2$. 
    We claim that the twisted action $\mathcal{P}_\ell(\Gamma)$ on $W_q(\Gamma)^h\times W_{q-1}(\Gamma)^{m-h}$ is not transitive.
    If $t\ge 1$, then $1\le e\le s$ and $1\le h\le m$ in Theorem~\ref{thm:arc-transitivity-new}.
    We have known in previous proof that if $\mathcal{P}_\ell(\Gamma)$ is transitive on $W_j(\Gamma)$, then $\mathcal{P}_\ell(\Gamma)$ is also transitive on $W_{j-1}(\Gamma)$ for $j\ge 1$.
    Therefore, if $\mathcal{P}_\ell(\Gamma)$ on $W_q(\Gamma)^h\times W_{q-1}(\Gamma)^{m-h}$ is transitive, then $\mathcal{P}_\ell(\Gamma)$ is transitive on $W_0(\Gamma)^m$ in the $+d$-twisted closed walk action.
    Let $x_1,x_2$ be different vertices of $\Gamma$.
    Then $(x_1,x_1,\ldots),(x_1,x_2,\ldots)\in W_0(\Gamma)^m$ where the skipped part can be any vertices.
    Then there is no $\alpha \in\Aut(\Gamma)$ such that
    $(x_1\alpha,x_1\alpha,\ldots)=(x_1,x_2,\ldots)$, a contradiction.
    This completes the proof.
\end{proof}

The next case we discuss directed Praeger--Xu graphs, which gives an alternative proof of \cite[Theorem 2.8 {\rm (iii)}]{Praeger1989high}

\begin{corollary}\label{cor:arc-transitivity-complete-loops}
    Let $m\ge 2$, $r\ge3$, $s\ge1$ and $t\ge0$. Then $\overrightarrow{\operatorname{PX}}(m,r,s)$ is $t$-arc-transitive if and only if $s+t\le r$. In particular, it is vertex-transitive if and only if $s\le r$.
\end{corollary}
\begin{proof}
    Let $\Gamma$ be the complete graph admitting self-loops of order $m$.
    Recall that $\overrightarrow{\operatorname{PX}}(m,r,s)=\overrightarrow{\operatorname{PX}}(\Gamma,r,s)$.
    Let $B_0=\{(x_1,\ldots,x_s,0)\in V\Sigma\mid x_i\in V\Gamma\}$
    and let
    \[L=\{(v_0,\ldots,v_t)\in W_t(\Sigma)\mid v_0\in B_0\}\]
    be the set of $t$-arcs whose starting vertex lies in $B_0$.
    Let
    \[A_0=\{(g_1,\ldots,g_r,0)\in\Aut(\Sigma)\}.
    \]
    As in the proof of Theorem~\ref{thm:arc-transitivity-new}, $\Sigma$ is $t$-arc-transitive if and only if $A_0$ acts on $L$ transitively.
    Moreover, every $l\in L$ can be represented uniquely by $s+t$ vertices $x_1,\ldots,x_{s+t}$ and the action of $A_0$ on $L$ is represented by
    \[\underline{x_1,\ldots,x_{s+t}}(g_1,\ldots,g_r,0)
        =\underline{x_1g_1,\ldots,x_{s+t}g_{s+t}}\]
    
    Note that every pair of permutations of $V\Gamma$ is a two-fold automorphism.
    Hence
    \[
    A_0=\operatorname{Sym}(V\Gamma)^r
    \]
    If $s+t\le r$, these permutations can be chosen independently. Since every sequence of vertices of $\Gamma$ satisfies the required arc conditions, the action is transitive.

    If $s+t>r$, the entries in positions $1$ and $r+1$ are acted on by the same permutation. Thus the action is not transitive. 
    The final assertion is the case $t=0$.
\end{proof}

For the loopless complete graph, the criterion is different.

\begin{corollary}\label{cor:arc-transitivity-complete-loopless}
    Let $\Gamma=K_n$ be the loopless complete graph on $n\ge3$ vertices, and let $r\ge3$, $s\ge1$ and $t\ge0$. Then $\overrightarrow{\operatorname{PX}}(K_n,r,s)$ is $t$-arc-transitive if and only if $s\mid r$ and $t\le s$. In particular, it is vertex-transitive if and only if $s\mid r$, and in that case its largest degree of arc-transitivity is exactly $s$.
\end{corollary}
\begin{proof}
    The graph $K_n$ is connected and non-bipartite, so $K_n\times\overrightarrow{K}_2$ is connected. If $(\alpha,\beta)$ is a two-fold automorphism of $K_n$, then it preserves both the arc relation and its complement. Since the only non-arcs of $K_n$ are the pairs $(v,v)$, the pair $(v\alpha,v\beta)$ must be a non-arc for every vertex $v$. Thus $v\alpha=v\beta$ for every $v$, and hence $\alpha=\beta$.

    By Corollary~\ref{cor:arc-transitivity-diagonal}, the digraph $\overrightarrow{\operatorname{PX}}(K_n,r,s)$ is $t$-arc-transitive if and only if $s\mid r$ and $\operatorname{Sym}(n)$ is transitive on $W_{\lceil t/s\rceil}(K_n)$. The natural action of $\operatorname{Sym}(n)$ is transitive on $W_0(K_n)$ and $W_1(K_n)$. For every $k\ge2$, choose distinct vertices $a,b,c$. The walks $(a,b,a)$ and $(a,b,c)$ extend to walks of length $k$, but no permutation maps one extension to the other. 
    Therefore, $\operatorname{Sym}(n)$ is transitive on $W_k(K_n)$ exactly when $k\le1$.

    It follows that the required condition is $s\mid r$ and $\lceil t/s\rceil\le1$, equivalently $s\mid r$ and $0\le t\le s$. The final claim follows by taking $t=0$ and $t=s$, respectively.
\end{proof}

\begin{corollary}\label{cor:arc-transitivity-complete-multipartite}
    Let $a\ge3$, $b\ge1$, and let $\Gamma=K_{b,\ldots,b}$ be the complete multipartite graph with $a$ parts of size $b$. Let $r\ge3$, $s\ge1$ and $t\ge0$. Then $\overrightarrow{\operatorname{PX}}(\Gamma,r,s)$ is $t$-arc-transitive if and only if $s\mid r$ and $t\le s$. In particular, it is $s$-arc-transitive if and only if $s\mid r$, and in that case its largest degree of arc-transitivity is exactly $s$.
\end{corollary}
\begin{proof}
    Write $q=\lceil g/s\rceil$, $d=\gcd(r,s)$, $\ell=r/d$ and $m=s/d$. By Example~\ref{examp:PX},
    \[
        \mathcal{P}_\ell(\Gamma)
        =\{(h_1\sigma^*,\ldots,h_\ell\sigma^*)\mid h_i\in H,\ \sigma\in\operatorname{S}_a\},
    \]
    where $H$ preserves every part of $\Gamma$.

    Suppose first that $s\nmid r$. Then $m\ge2$. 
    We have known in previous proof that if $\mathcal{P}_\ell(\Gamma)$ is transitive on $W_j(\Gamma)$, then $\mathcal{P}_\ell(\Gamma)$ is also transitive on $W_{j-1}(\Gamma)$ for $j\ge 1$.
    Therefore, if $\mathcal{P}_\ell(\Gamma)$ on $W_q(\Gamma)^h\times W_{q-1}(\Gamma)^{m-h}$ is transitive, then $\mathcal{P}_\ell(\Gamma)$ is transitive on $W_0(\Gamma)^{m}$ in the $+d$-twisted closed walk action. 
    Let $x_1,x_2\in V\Gamma$ such that $x_1$ and $x_2$ lie in different part of $\Gamma$.
    Then $(x_1,x_1,\ldots),(x_1,x_2,\ldots)\in W_0(\Gamma)^m$ where the skipped parts can be any vertices.
    However, there is no $(h_1\sigma^*,\ldots,h_\ell\sigma^*)\in\mathcal{P}_\ell(\Gamma)$ such that $(x_1h_1\sigma^*,x_1h_2\sigma^*,\ldots)=(x_1,x_2,\ldots)$ as $x_1$ and $x_2$ lies in different parts of $\Gamma$.
    By Theorem~\ref{thm:arc-transitivity-new}, $\overrightarrow{\operatorname{PX}}(\Gamma,r,s)$ is not $t$-arc-transitive for any $t\ge0$.

    Now suppose that $s\mid r$. Then $d=s$ and $m=1$. The group $\mathcal{P}_\ell(\Gamma)$ is transitive on $W_0(\Gamma)$ and $W_1(\Gamma)$. 
    Therefore, Theorem~\ref{thm:arc-transitivity-new} shows that $\overrightarrow{\operatorname{PX}}(\Gamma,r,s)$ is $t$-arc-transitive for $0\le t\le s$.

    It remains to consider $t>s$. In this case $q=\lceil t/s\rceil\ge2$ in Theorem~\ref{thm:arc-transitivity-new}. 
    The closed walk action of $\mathcal{P}_\ell(\Gamma)$ on $W_2(\Gamma)$ is not transitive: choose vertices $x_0,x_1,x_2$ in parts $0,1,0$, respectively, and vertices $y_0,y_1,y_2$ in parts $0,1,2$, respectively. Then $(x_0,x_1,x_2)$ and $(y_0,y_1,y_2)$ lie in $W_2(\Gamma)$, but they cannot be in the same orbit. 
    Therefore, $\mathcal{P}_\ell(\Gamma)$ is not transitive on $W_2(\Gamma)$ which implies it is not transitive on $W_q(\Gamma)$ as well. 
    Theorem~\ref{thm:arc-transitivity-new} now shows that $\overrightarrow{\operatorname{PX}}(\Gamma,r,s)$ is not $t$-arc-transitive. This proves the result.
\end{proof}

\subsection{Cayleyness}
In this subsection, we explore when $\overrightarrow{\operatorname{PX}}(\Gamma,r,s)$ is a Cayley digraph.

\begin{theorem}\label{thm:Cayleyness}
    Let $\Gamma$ be a digraph such that $\Gamma\times\overrightarrow{K}_2$ is connected, and let $s\ge 1$ and $r\ge 3$. 
    Let $\Sigma=\overrightarrow{\operatorname{PX}}(\Gamma,r,s)$.
    \begin{itemize}
        \item [\rm (i)] $\Gamma[r]$ is a Cayley graph if and only if $\Gamma$ is a Cayley graph.
        \item [\rm (ii)] If every two-fold automorphism of $\Gamma$ has the form $(\alpha,\alpha)$ with $\alpha\in\Aut(\Gamma)$, then $\Sigma$ is a Cayley graph if and only if $s\mid r$ and $\Gamma$ is a Cayley graph.
    \end{itemize}
\end{theorem}
\begin{proof}
    Let $G=\Aut(\Gamma)$. First suppose that $|V\Gamma|>1$.
    By Theorem~\ref{thm:aut-general-px}, the action of $\Aut(\Gamma[r])=G\wr\mathbb{Z}_r$ on $V\Gamma[r]$ is the cycle-shift action.
    {\rm (i)} follows from 
    Lemma~\ref{lem:cycle-shift}~{\rm (iii)}, with $s=r$.

    Now assume that every two-fold automorphism of $\Gamma$ is diagonal, and suppose that $\Sigma$ is a Cayley graph. 
    If $|V\Gamma|>1$, then $\Sigma$ is vertex-transitive, so Corollary~\ref{cor:arc-transitivity-diagonal}, with $t=0$, implies that $G$ is transitive and $s\mid r$. 
    By Theorem~\ref{thm:aut-general-px}, a regular subgroup of $\Aut(\Sigma)$ is also a regular subgroup of $G\wr\mathbb{Z}_r$ in its cycle-shift action. Lemma~\ref{lem:cycle-shift}~{\rm (iii)} implies that $G$ has a regular subgroup. Hence $\Gamma$ is a Cayley graph. The same conclusion is immediate when $|V\Gamma|=1$.

    Conversely, suppose that $\Gamma$ is a Cayley graph, and let $R\leq G$ be regular on $V\Gamma$. Define
    \[
        L=\{(g_1,\ldots,g_r,\emptycommand{j})\in R\wr\mathbb{Z}_r\mid g_{i+s}=g_i\text{ for all }i\in\mathbb{Z}_r\}.
    \]
    Since $s\mid r$, $L$ is a subgroup. Moreover, $g_i=g_{i+s}$ implies $g_i\to g_{i+s}$, so Theorem~\ref{thm:aut-general-px} gives $L\leq\Aut(\Sigma)$.
    Let
    \[
        u=(x_1,\ldots,x_s,\emptycommand{i})
        \quad\text{and}\quad
        v=(y_1,\ldots,y_s,\emptycommand{j})
    \]
    be vertices of $\Sigma$. For each $a\in\{1,\ldots,s\}$, choose the unique element $g_{i+a}\in R$ such that $x_ag_{i+a}=y_a$. These elements determine an element
    $(g_1,\ldots,g_r,\emptycommand{j-i})$ of $L$ sending $u$ to $v$. Thus $L$ is transitive on $V\Sigma$.
    Also note that $|L|=|R|^s\cdot r=|\Sigma|$.
    Therefore $L$ is regular, so $\Sigma$ is a Cayley graph. This proves~{\rm (ii)}.
\end{proof}

\section{Examples of graphs of exponential growth}\label{sec:exponential-growth}

This section introduces an infinite family of graphs where the order of a vertex stabilizer increases exponentially with respect to the order of the graph.

By \cite{Praeger1989high}, if a vertex-transitive digraph $\Gamma$ is not twin-free, then there are digraph $\Sigma$ and integer $n$ such that $\Gamma=\Sigma[\overline{K_n}]$. Therefore, $\Aut(\Gamma)\cong\S_n\wr_{V\Sigma}\Aut(\Sigma)$. Let the order of a vertex stabilizer of $\Sigma$ be $O$. Then the order of a vertex stabilizer of $\Gamma$ is $O\cdot(n-1)!\cdot(n!)^{|V\Sigma|-1}$, which is large. Our example is a family of twin-free digraphs.

\begin{theorem}\label{thm:exponential-growth}
    Let $r\ge 2$, let $\Gamma=\operatorname{PX}(2,3,1)$, and let $\Sigma=\overrightarrow{\operatorname{PX}}(\Gamma,2r,2)$.
    \begin{itemize}
        \item [\rm (i)] $\Aut(\Sigma)\cong (\mathbb{Z}_2^{3r}\rtimes \S_3)^2\rtimes\mathbb{Z}_{2r}$.
        \item [\rm (ii)] $\Sigma$ is a twin-free connected arc-transitive digraph of order $72r$.
        \item [\rm (iii)] A vertex stabilizer of $\Aut(\Sigma)$ is of order $2^{6r}$.
        \item [\rm (iv)] $\Sigma$ is of out-valency $4$ and is not isomorphic to a directed Praeger-Xu graph.
    \end{itemize}
\end{theorem}
\begin{proof}
	By Example~\ref{examp:PX}, $\Gamma\cong K_{2,2,2}$ and
	\[
		\Aut^{\operatorname{TF}}(\Gamma)=(H\times H)\rtimes \S_3,
		\qquad H\cong\mathbb{Z}_2^3.
	\]
	By
	Theorem~\ref{thm:aut-general-px}, let 
	\[L=\{(g_1,\ldots,g_{2r},0)\mid g_i\to g_{i+2}\}\le\Aut(\Sigma).\]
	Here the subscript $i$ is taken modulo $2r$.
	The map $i\mapsto i+2$ has two orbits, each of length $r$. 
    Recall that Example~\ref{examp:PX} states that $(\alpha,\beta)\in\Aut^{\operatorname{TF}}(\Gamma)$ if and only if $\alpha=h_1\sigma^*$ and $\beta=h_2\sigma^*$ for some $h_1,h_2\in H$ with the same $\sigma$.
    Hence a closed walk $g_1\to g_3\to \cdots\to g_{2r-1}\to g_1$ consists of $h_j\sigma^*$ with a common $\sigma$.
	Therefore, on either orbit, the preceding description
	of $\Aut^{\operatorname{TF}}(\Gamma)$ shows that
	\[
		L\cong(\mathbb{Z}_2^{3r}\rtimes\S_3)^2.
	\]
	Thus
	\[
		\Aut(\Sigma)=L\rtimes\mathbb{Z}_{2r}
		\cong(\mathbb{Z}_2^{3r}\rtimes\S_3)^2\rtimes\mathbb{Z}_{2r},
	\]
	proving~{\rm(i)}.

	The graph $\Gamma\cong K_{2,2,2}$ is connected and non-bipartite, so
	$\Gamma\times\overrightarrow{K}_2$ is connected. Then by Theorem~\ref{cor:nice-connectivity}, $\Sigma$ is connected. 
    By Lemma~\ref{lem:loosely-attached}, $\Sigma$ is loosely-attached, hence is also twin-free.
    As $\Gamma$ is arc-transitive, by Theorem~\ref{thm:ver-trans} and Theorem~\ref{thm:arc-transitivity-apply}, $\Sigma$ is both vertex-transitive and arc-transitive.
	
	Finally,
	\[
		|V\Sigma|=|V\Gamma|^2(2r)=6^2(2r)=72r,
	\]
	which proves~{\rm(ii)}.

	By~{\rm(i)},
	\[
		|\Aut(\Sigma)|=(2^{3r}\cdot 6)^2(2r).
	\]
	Since $\Sigma$ is vertex-transitive, the orbit-stabilizer theorem and
	part~{\rm(ii)} give
	\[
		|\Aut(\Sigma)_v|
		=\frac{(2^{3r}\cdot 6)^2(2r)}{72r}=2^{6r}
	\]
	for every vertex $v$ of $\Sigma$, proving~{\rm(iii)}.

    Finally, assume that $\Sigma\cong\overrightarrow{\operatorname{PX}}(m',r',s')$. Since $\Sigma$ is vertex-transitive, $s'\le r'$. Note that $\operatorname{val}(\Gamma)=4$.
    Recall that the out-valency of $\overrightarrow{\operatorname{PX}}(\Gamma,r,s)$ equals that of $\Gamma$.
    We have that the out-valency of $\Sigma$ is also $4$.
    Therefore, $m'=4$.
    Then by Lemma~\ref{lem:auto-of-De}, $\Aut(\Sigma)=\S_4\wr\mathbb{Z}_{r'}$ and
    $\Aut(\Sigma)_v=(\S_3)^{s'}\times (\S_4)^{r'-s'}$.
    However, $\Aut(\Sigma)_v$ is a $2$-group by {\rm (iii)}, which is impossible. This proves {\rm (iv)}.
\end{proof}

\begin{remark}\label{rmk:conjecture}
    Let $X$ be the underlying graph $\operatorname{PX}(\Gamma,2r,2)$ of $\Sigma$ in Theorem~\ref{thm:exponential-growth}. Then $X$ is $8$-valent.  Thus, this example is related to Conjecture~3.1 and Question~3.2 of Verret~\cite{Verret2013Order}, which concern undirected graphs.

    Write $V\Gamma=\mathbb{Z}_2\times\mathbb{Z}_3$. The permutation
    \[
        \iota\colon (x,y,i)\longmapsto (y,x,-i)
    \]
    reverses the arcs of $\Sigma$. It normalizes $\Aut(\Sigma)$, and hence
    \[
        G=\langle\Aut(\Sigma),\iota\rangle
        =\Aut(\Sigma)\rtimes\langle\iota\rangle
    \]
    is arc-transitive on $X$. Let $N$ be the subgroup of $L$ in the proof of Theorem~\ref{thm:exponential-growth} obtained by taking all the $\mathbb{Z}_2^3$-components and trivial $\S_3$-components. Then
    \[
        N\cong\mathbb{Z}_2^{6r}\trianglelefteq G.
    \]
    Let $(u,v)$ be an arc of $\Sigma$. One can check that
    \[
        |N_{uv}|=2^{6r-3}>1.
    \]
    Thus $N$ is an elementary abelian normal $2$-subgroup of $G$ which is not semiregular on the arcs of $X$.

    Moreover, $|G_{uv}|=2^{6r-2}$, so these soluble arc-stabilizers grow exponentially with $|V X|=72r$. Consequently, the family realizes the exceptional alternative in Verret's conjecture and gives an explicit infinite family of the type asked about in his Question~3.2. It does not, however, give a general classification as requested there.
\end{remark}

\section{Direct product that is a Cayley digraph}\label{sec:direct}

In this section, we give an example of a pair of digraphs $(\Sigma_1,\Sigma_2)$ such that $\Sigma_1\times\Sigma_2$ is Cayley, but $\Sigma_1$ is not vertex-transitive.
\begin{theorem}\label{thm:auto-g-de-bruijn}
  Let $\Gamma$ be a digraph such that $\Gamma\times\overrightarrow{K}_2$ is connected.
  Then \[\Aut(\Gamma^{*n})=\{(\alpha,\ldots,\alpha)\in \operatorname{Sym}(V\Gamma^n)\mid \alpha\in\Aut(\Gamma)\}\cong\Aut(\Gamma).\]
\end{theorem}
\begin{proof}
    Recall that $V\Gamma^{*n}=(V\Gamma)^n$.
  For every $\alpha\in\Aut(\Gamma)$, let $\alpha^{(n)}=(\alpha,\ldots,\alpha)$.
  Clearly, $\alpha^{(n)}\in\Aut(\Gamma^{*n})$.
  Thus we have \[\Aut(\Gamma^{*n})\supseteq\{(\alpha,\ldots,\alpha)\in \operatorname{Sym}(V\Gamma)\mid \alpha\in\Aut(\Gamma)\}.\]

  For the reverse inclusion, let $\beta\in\Aut(\Gamma^{*n})$.
  Let $\mathbf{x}=(x,v_2,\ldots,v_n)$ and $\mathbf{y}=(v_2,\ldots,v_n,y)$ be vertices of $\Gamma^{*n}$ such that $(\mathbf{x},\mathbf{y})\in A\Gamma^{*n}$.
  Then $(x,y)\in A\Gamma$ and $\mathbf{x}\beta\to\mathbf{y}\beta$.
  For each $i\in\{1,\ldots,n\}$, let $\pi_i:(V\Gamma)^n\to V\Gamma$ be the natural projection onto the $i$th coordinate.
  Then
  \[
    \pi_i((x,v_2,\ldots,v_n)\beta)=\pi_{i-1}((v_2,\ldots,v_n,y)\beta)
  \]
  for all $i\in\{2,\ldots,n\}$ and pair $(x,y)\in A\Gamma$.

  Since $\Gamma\times\overrightarrow{K}_2$ is connected, it follows that for every $w_1,\ldots,w_n,x\in V\Gamma$ and every $i\in\{2,\ldots,n\}$ we have
  \[
    \pi_i((w_1,\ldots,w_n)\beta)=\pi_{i-1}((w_2,\ldots,w_n,x)\beta)\tag{*}\label{equ:a_equation}.
  \]
  In particular, for every $w_1,\ldots,w_n,x_1,\ldots,x_{n-1}\in V\Gamma$,
  \[
    \pi_n((w_1,\ldots,w_n)\beta)=\pi_1((w_n,x_1,\ldots,x_{n-1})\beta).
  \]
  Hence $\pi_n((w_1,\ldots,w_n)\beta)$ depends only on $w_n$.
  Define $\alpha:V\Gamma\rightarrow V\Gamma$ by
  \[
    w\alpha:=\pi_n((w_1,\ldots,w_{n-1},w)\beta),
  \]
  where $w_1,\ldots,w_{n-1}\in V\Gamma$ are arbitrary.
  The previous observation shows that $\alpha$ is well-defined.

  Now for any $w_1,\ldots,w_n\in V\Gamma$ and any $i\in\{1,\ldots,n\}$, repeatedly applying (\ref{equ:a_equation}) above gives
  \[
    \pi_i((w_1,\ldots,w_n)\beta)=w_i\alpha.
  \]
  Therefore,
  \[
    (w_1,\ldots,w_n)\beta=(w_1\alpha,\ldots,w_n\alpha).
  \]
  It implies that $\alpha$ is bijective, hence $\alpha\in\Sym(V\Gamma)$.

  Finally, since $\beta$ preserves arcs of $\Gamma^{*n}$, the defining condition of $A\Gamma^{*n}$ implies
  \[
    (u,v)\in A\Gamma\iff (u\alpha,v\alpha)\in A\Gamma.
  \]
  Hence $\alpha\in\Aut(\Gamma)$ and $\beta=\alpha^{(n)}$.
Therefore $\Aut(\Gamma^{*n})=\{\alpha^{(n)}\mid \alpha\in\Aut(\Gamma)\}$ and $\Aut(\Gamma^{*n})\cong\Aut(\Gamma)$.
\end{proof}

\begin{remark}
    By Theorem~\ref{thm:auto-g-de-bruijn}, it is notable that the orbits of $\Aut(\Gamma^{*2})$ on $V\Gamma^{*2}$ is coincident with the orbital of $\Aut(\Gamma)$.
    It would be interesting to understand the relationship between $\Gamma^{*2}$ and orbital graphs of $\Gamma$.
\end{remark}

Now we are ready to give the example.

\begin{example}\label{examp:non-Cayley}
    Let $\Gamma$ be a Cayley digraph such that $\Gamma\times\overrightarrow{K}_2$ is connected. For example, take $\Gamma=C_3$, the undirected cycle of length $3$.
    Let $r\ge \max\{3,\log_{|V\Gamma|}(|\Aut(\Gamma)|)+1\}$ be an integer.
    Let $\Sigma_1=\Gamma^{*r}$ and $\Sigma_2=\overrightarrow{C}_r$.
    Then
    $|V\Sigma_1|=|V\Gamma|^r>|\Aut(\Gamma)|$.
    By Theorem~\ref{thm:auto-g-de-bruijn}, $\Aut(\Gamma)\cong\Aut(\Sigma_1)$.
    This implies that $|V\Sigma_1|>|\Aut(\Sigma_1)|$.
    Hence $\Sigma_1$ is not vertex-transitive.
    By Lemma~\ref{lem:cycle-shift}~{\rm (iii)} and Theorem~\ref{thm:aut-general-px}, $\Sigma_1\times\Sigma_2=\Gamma[r]$ is a Cayley digraph as $\Gamma$ is a Cayley digraph.
    This gives the example.
\end{example}

	\vskip0.2in

    \noindent{\bf Declaration of Generative AI and AI-Assisted Technologies.}
    During the preparation of this manuscript, the authors used ChatGPT Terra 5.6 for English language editing, ChatGPT Sol 6.1 to assist in drafting the proof of Theorem~\ref{thm:arc-transitivity-apply} and Corollary~\ref{cor:arc-transitivity-complete-multipartite}, and Claude 5.1 to assist in sanity-checking intermediate mathematical arguments and drafting verification scripts. 
    All core mathematical ideas originated solely with the authors. 
    The authors accept full responsibility for the mathematical validity and accuracy of all results presented in this paper.
    
    \vskip0.2in
	\noindent\thanks{{\bf Acknowledgments.}
		This work was supported by the National Natural Science Foundation of China (No.11901512) and the Natural Science Foundation of Yunnan Province (No. 202401AT070275). 
        We are grateful for Gabreil Verret's valuable comments and advices.}

\end{document}